\documentclass[11pt, oneside]{article}   	
\usepackage{amsmath, amsthm, amsfonts, amssymb}
\usepackage{graphicx}
\usepackage[colorlinks,citecolor=blue,urlcolor=blue]{hyperref}
\newcommand\E{{\mathcal E}}
\newcommand\Ex{{\mathbb E}}

\newcommand\Prob{{\mathbb P}}

\newcommand\cP{{\mathcal P}}

\newcommand\N{{\mathbb N}}

\newcommand\R{{\mathbb R}}
\newcommand\C{{\mathbb C}}

\newcommand\one{{\bf 1}}

\newcommand\norm[1]{\|#1\|}

\newcommand\bra[1]{\langle #1 \rangle}

\newtheorem{theorem}{Theorem}[section]
\newtheorem{corollary}[theorem]{Corollary}
\newtheorem{lemma}[theorem]{Lemma}
\newtheorem{proposition}[theorem]{Proposition}

\theoremstyle{definition}

\theoremstyle{remark}
\newtheorem{remark}[theorem]{Remark}

\title{Poisson statistics for beta ensembles on the real line at high temperature}
\author{Fumihiko Nakano\footnote{Department of Mathematics, Gakushuin University, Tokyo, Japan. \newline Email: fumihiko@math.gakushuin.ac.jp  \newline
\emph{Current address:} Mathematical Institute, Tohoku University, Sendai,  Japan
\newline Email: fumihiko.nakano.e4@tohoku.ac.jp} \and Khanh Duy Trinh\footnote{Global Center for Science and Engineering, Waseda University, Tokyo, Japan. \newline
Email: trinh@aoni.waseda.jp}}

\begin{document}
\maketitle

\begin{abstract}
This paper studies beta ensembles on the real line in a high temperature regime, that is, the regime where $\beta N \to const \in (0, \infty)$, with $N$ the system size and $\beta$ the inverse temperature. For the global behavior, the convergence to the equilibrium measure is a consequence of a recent result on large deviation principle. This paper focuses on the local behavior and shows that the local statistics around any fixed reference energy converges weakly to a homogeneous Poisson point process. 

% by Garc\'a-Zelada (Ann.\ Inst.\ Henri Poincar\'e Probab.\ Stat.\ (2019)) and by Liu and Wu (Stochastic Processes and their Applications (2020)). 
\medskip

	\noindent{\bf Keywords: } beta ensembles ; high temperature ; large deviation principle ; Poisson statistics
		
\medskip
	
	\noindent{\bf AMS Subject Classification: } Primary 60F05; Secondary  60B20,  60G55
\end{abstract}

\section{Introduction}
Let $V \colon \R \to \R$ be a measurable function. Let 
\begin{align*}
	H_N = H_N(\lambda_1, \dots, \lambda_N) &= \frac1N \sum_{i = 1}^N V(\lambda_i) - \frac1{N^2}\sum_{i \neq j} \log|\lambda_j - \lambda_i| 
\end{align*}
be the energy of the configuration $(\lambda_1, \lambda_2, \dots, \lambda_N) \in \R^N$ under the external potential $V$ and the log-interaction. Beta ensembles are then defined as ensembles of $N$ particles with the joint probability density function propositional to
\[
	 e^{-\frac{\beta N^2}{2} H_N} = |\Delta (\lambda)|^\beta e^{-\frac{\beta N}{2} \sum_{i = 1}^N V(\lambda_i)}.
\]
Here $\Delta(\lambda) = \prod_{i < j}(\lambda_j - \lambda_i)$ is the Vandermonde determinant. The parameter $\beta>0$ is regarded as the inverse temperature of the system.

When all $\{\lambda_i\}$ are distinct, the energy functional can be expressed as 
\[
	H_N=\int V(x) dL_N(x) - \iint_{x\neq y} \log|x-y|dL_N(x) dL_N(y),
\]
where $L_N = N^{-1}\sum_{i = 1}^N \delta_{\lambda_i}$ denotes the empirical distribution with $\delta_\lambda$ the Dirac measure. Then under some mild conditions on $V$, the energy functional
\begin{align*}
	\E (\mu) &= \iint \left(\frac12 V(x) + \frac12 V(y) - \log|x-y| \right) d\mu(x) d\mu(y) \\
	&= \int V(x) d\mu(x) -\iint \log|x-y|d\mu(x) d\mu(y)
\end{align*}
which is well-defined on the set $\cP(\R)$ of probability measures on $\R$ has a unique minimizer $\mu_V$ of compact support, that is,
\[
	\E(\mu_V)  = \inf_{\mu \in \cP(\R)} \E(\mu). 
\]
The minimizer $\mu_V$ is an equilibrium of the system in the sense that for fixed $\beta > 0$, as $N \to \infty$, the empirical distribution $L_N$ converges weakly to $\mu_V$, almost surely. This result, together with Gaussian fluctuations around the limit can be found in \cite{Johansson-1998}.

We would like to study the case where the parameter $\beta$ varies as $N$ tends to infinity. For the global behavior, that is, the limiting behavior of the sequence of the empirical measures $\{L_N\}$, following facts are known \cite{Chafai-2014, G-Zelada-2019, Liu-Wu-2019}: as $\beta N \to 2c \in (0, \infty]$, the sequence of $\{L_N\}$ satisfies a large deviation principle (LDP), and thus converges weakly to a limiting measure $\mu_c$, almost surely. If $c=\infty$, the limiting measure coincides with that for the case of fixed $\beta$, that is, $\mu_\infty = \mu_V$. Gaussian fluctuations around the limit in case of varying parameter $\beta$ have been studied for some specific models: Gaussian beta ensembles ($V(x) = x^2/2$) \cite{Trinh-2019}, beta Laguerre ensembles \cite{Trinh-Trinh-2019} and circular beta ensembles \cite{Hardy-Lambert-2019}. In any case, the global behavior is governed by $\lim_{N \to \infty} \beta N$. In contrast, for the local or edge scaling limit, the limits are $\rm{Sine}_\beta$ point processes and $TW_\beta$ distributions depending on $\beta$ \cite{BEY-2012, BEY-2014, BEY-2014E, Ramirez-Rider-Virag-2011, Valko-Virag-2009}.

Gaussian beta ensembles realized as eigenvalues of a tridiagonal random matrix model \cite{DE02} are among the most studied models. Let us only mention some of their results in a high temperature regime, the regime where $\beta N \to 2c \in (0, \infty)$. For the global behavior, the limiting measures $\mu_c$ were explicitly calculated in \cite{Allez12, DS15}. They are Gaussian like probability measures of full-support which are (up to a scaling) probability measures of associated Hermite polynomials. Under suitable scaling, they provide an interpolation between the semi-circle distribution ($\mu_\infty$ in this case) and the standard Gaussian distribution (the probability measure with density proportional to $e^{-V(x)} = e^{-x^2/2}$ associated with the potential $V$). The almost sure convergence of the sequence of empirical distributions and Gaussian fluctuations around the limit were established in \cite{Trinh-2019} by using the random matrix model. Next, for the local behavior, it was shown in \cite{Peche-2015, Nakano-Trinh-2018} that the local statistics around any fixed point converges to a homogeneous Poisson point process on $\R$. Furthermore, the edge behavior (of a scaled model like~\eqref{VbE} below when $\beta N \log N \to 0$) was studied in \cite{Pakzad-2018}. The main purpose of this paper is to show the universality of the local Poisson behavior for generic potential $V$. We remark that the results here are extended to more general cases and the edge limit is established in the regime $\beta N \to 2c \in (0, \infty)$ in \cite{Lambert-2019}.

From now on, let us consider the following beta ensembles 
\begin{equation}\label{VbE}
	(\lambda_1, \lambda_2, \dots, \lambda_N) \propto \frac{1}{Z_{\beta, N}} |\Delta(\lambda)|^\beta e^{-\sum_{i = 1}^N V(\lambda_i)},
\end{equation}
in the regime where $\beta N \to 2c \in (0, \infty)$. Here $Z_{\beta, N}$ is the normalizing constant. In this regime, to match the ensembles at the beginning, one can replace $V$ by $cV$. A LDP has been recently studied for more general models  \cite{G-Zelada-2019,Liu-Wu-2019}. It turns out that when the potential $V$ is bounded below and the following moments condition is satisfied
\begin{equation}\label{condition-V}
	\int |x|^k e^{-V(x)} dx < \infty, \quad k = 0,1,2,\dots,
\end{equation}
then the sequence of empirical distributions $\{L_N\}$ satisfies a LDP with the good rate function $I_c (\mu) = \E_c(\mu) - \inf_{\nu} \E_c(\nu)$, where the functional $\E_c$ is defined for absolutely continuous probability measure $\mu(dx) = \rho(x) dx$,
\[
	\E_c(\mu) = \int \log(\rho(x)) \rho(x) dx + \int V(x) \rho(x) dx  - c \int \log|x - y| \rho(x) \rho(y) dx dy.
\]
(See Section~2 for a more precise definition.) When the potential $V$ is assumed to be lower semi-continuous, we only need weaker conditions \cite{G-Zelada-2019}.
Such functional has appeared in heuristic saddle point arguments as in \cite{Akemann-Byun-2019, Allez12, Spohn-2019}. Note that the functional $\E_c$ has a unique minimizer, denoted by $\mu_c$ (or $\rho_c$ for the density), because of the strict convexity. Then the LDP implies the almost sure convergence of empirical distributions to the equilibrium measure $\mu_c$. Namely, the following results hold.

\begin{theorem}[{\cite{G-Zelada-2019, Liu-Wu-2019}}]\label{thm:LLN-intro}
	Assume that the function $V$ is bounded below and satisfies the moments condition \eqref{condition-V}. Then in the regime where $\beta N \to 2c \in (0, \infty)$, the following hold.
	\begin{itemize}
		\item[\rm (i)] The sequence of empirical distributions $\{L_N\}$ satisfies a LDP with the good rate function $I_c$. 
		\item[\rm (ii)] The function $I_c$ is strictly convex and has a unique minimizer $\mu_c$ which is absolutely continuous.
		\item[\rm (iii)] The sequence $\{L_N\}$ converges weakly to $\mu_c$, almost surely.
	\end{itemize}
\end{theorem}

This paper focuses on studying the limiting behavior of the local statistics around a fixed reference energy $E \in \R$,
\[
	\xi_N(E) = \sum_{i = 1}^N \delta_{N(\lambda_i - E)}.
\]
We show that the local statistics $\xi_N(E)$ converges weakly to a homogeneous Poisson point process on $\R$. That local behavior in the case of Gaussian beta ensembles was proved in \cite{Peche-2015} and in \cite{Nakano-Trinh-2018} by different methods. However, the two approaches relied more or less on both the joint density and the tridiagonal matrix model. This paper refines ideas developed in the two papers to extend the result to the case of generic potential $V$. Our main result is stated as follows.

\begin{theorem}\label{main-result}
	Assume that the potential $V$ is continuous and that 
	\[
		\lim_{x \to \pm \infty} \frac{V(x)}{\log (1+x^2)} = \infty.
	\]
Then the following hold. 
\begin{itemize}

\item[\rm(i)]
The continuous version of $\rho_c$ satisfies the relation 
\[
\rho_c(x) = \frac{1}{Z_{c} }e^{-V(x) + 2 c  \int \log|x - y| \rho_c(y) dy}, \quad \text{for all $x \in \R$,}
\]
where $Z_c$ is a constant.

\item[\rm(ii)]
For fixed $E \in \R$, the local statistics $\xi_N(E)$
converges weakly to a homogeneous Poisson point process on $\R$ with density $\rho_c(E) >0$.
\end{itemize}
\end{theorem}

\begin{remark}
	By a heuristic saddle point argument, it was also shown in  \cite{Akemann-Byun-2019} that the limiting measure is a minimizer of the energy functional $H_c$. Then by using functional derivative, an equation to characterize the minimizer $\mu_c$, with density $\rho_c$, was derived 
\[
	 \int \frac{V'(x) \rho_c(x)}{x - z} dx + c S_c^2(z) + S_c'(z) = 0,\quad z \in \C \setminus \R,
\]
where $S_c(z) = \int \frac{\rho_c(x)dx}{ x - z}$ is the Stieltjes transform of $\rho_c$. Note that for Gaussian beta ensembles ($V'(x) = x$), the integral in the above equation is equal to $1 + z S_c(z)$, and hence we can solve $S_c(z)$ and then get an explicit formula for $\rho_c$.

\end{remark}

The paper is organized as follows. The next section is devoted to introduce a LDP. Section~3 studies properties of the limiting measure $\mu_c$ and proves some estimates needed for Section 4 in which the Poisson statistics is derived.

\section{Large deviation principle}
\subsection{Assumption on the potential $V$}
Throughout this paper, we assume that the potential $V \colon \R \to \R$ is measurable, bounded below and 
\[
	\lim_{x \to \pm \infty} \frac{V(x)}{\log(1+x^2)} = \infty.
\]
Under that assumption, it is clear that the moments condition~\eqref{condition-V} is satisfied.
%\[
%\int_\R |x|^k e^{-V(x)} dx < \infty, \quad \text{for all $k = 0,1,2,\dots.$}
%\]
Let $\alpha$ be the probability measure with density $\alpha(x) = Z^{-1}e^{-V(x)}$, where $Z = \int_\R  e^{-V(x)} dx$. Then all moments of $\alpha$ are finite.

%\subsection{Energy functionals}
\subsection{Large deviation principle}
We give here a quick review on the results stated in Theorem~\ref{thm:LLN-intro}.
Let $\cP(\R)$ be the set of probability measures on $\R$, endowed with the weak topology. For $\mu, \nu \in \cP(\R)$, the entropy of $\mu$ relative to $\nu$ (also called the Kullback--Leibler divergence) is defined by 
\begin{equation}
	H(\mu | \nu) = \begin{cases}
		\int_\R \frac{d\mu}{d\nu} \log \frac{d\mu}{d\nu} d\nu, &\text{if $\mu \ll \nu,$} \\
		+\infty, &\text{otherwise.}
	\end{cases}
\end{equation}
Here $\mu \ll \nu$ means that the measure $\mu$ is absolutely continuous with respect to the measure $\nu$, and $\frac{d\mu}{d\nu}$ denotes the Radon--Nikodym derivative. It is known that  $H(\mu | \alpha)$ is non-negative and strictly convex on the sublevel set $\{\mu \in \cP(\R) : H(\mu | \alpha) \le L\}$, for any $L>0$ (see \cite[\S 6.2]{Dembo-Zeitouni-book}).

%For given $\nu$, the function $H(\mu | \nu)$ is a good rate function in the theory of large deviation principle, (see \cite[\S 6.2]{Dembo-Zeitouni-book} for example), that is,
%\begin{itemize}
%	\item[(i)]	$H(\mu | \nu) \ge 0$, (with equality only when $\mu = \nu$);
%	\item[(ii)] $H(\mu | \nu)$ is lower semi-continuous; and
%	\item[(iii)] for $L > 0$, the sublevel set $\{\mu \in \cP(\R) : H(\mu | \nu) \le L\}$ is compact.
%\end{itemize}
%In addition, it is also known that 
%\begin{itemize}
%	\item[(iv)] $H(\mu | \nu)$ is strictly convex on the sublevel set $\{\mu \in \cP(\R) : H(\mu | \nu) \le L\}$, for any $L>0$.
%\end{itemize}

%Let $\{X_i\}_{i}$ be an i.i.d.~(independent identically distributed) sequence of random variables with common distribution $\alpha$. Denote by  $\eta_N$ the empirical measure of $\{X_i\}_{i = 1}^N$,
%\[
%	\eta_N = \frac{1}{N} \sum_{i = 1}^N \delta_{X_i}.
%\]
%Then Sanov's theorem states that $\{\eta_N\}_N$ satisfies a LDP on $\cP(\R)$ with the good rate function $H(\mu | \alpha)$. Moreover, the function $H(\mu | \alpha)$ is strictly convex on the sublevel set $\{\mu \in \cP(\R) : H(\mu | \alpha) \le L\}$, for any $L>0$ (see \cite[\S 6.2]{Dembo-Zeitouni-book}). Since $\alpha$ has density $Z^{-1}e^{-V(x)}$, we can express the functional $H(\mu | \alpha)$ as follows
%\begin{equation}
%	H(\mu | \alpha) =  \int \rho(x) \log (\rho(x)) dx + \int V(x) \rho(x) dx + \log Z,
%\end{equation}
%where $\rho$ is the density of $\mu$, provided that $H(\mu | \alpha) < \infty$.

For $c > 0$, let 
\begin{equation}
	H_c(\mu) := 
	\begin{cases}
		H(\mu | \alpha) - c \iint \log|x-y| d\mu(x) d\mu(y), &\text{if $H(\mu|\alpha) < \infty$,} \\
		+\infty, &\text{otherwise.}
	\end{cases}
\end{equation}
The functional $H_c$ is well-defined, that is, $H_c(\mu) \in (-\infty, \infty)$, if $H(\mu|\alpha) < \infty$ (see Remark~2.3 in \cite{Liu-Wu-2019}). Also in case $H(\mu|\alpha) < \infty$, by the assumption on $V(x)$, we can use the Donsker--Varadhan variational formula to see that
\[
\int \frac12 V(x) d\mu(x)  \le H(\mu | \alpha) + \log \int e^{\frac12 V(x)} Z^{-1} e^{-V(x)} dx < \infty.
\] 
Thus, once the functional $H(\mu | \alpha)$ is finite, it can be expressed as 
\begin{equation*}
	H(\mu | \alpha) =  \int \rho(x) \log (\rho(x)) dx + \int V(x) \rho(x) dx + \log Z,
\end{equation*}
where $\rho$ is the density of $\mu$. In this case, the functional $H_c(\mu)$ is also finite and $H_c(\mu)$ can be written as a sum of finite integrals 
\begin{align}
	H_c(\mu) = H_c(\rho) &=  \int \rho(x) \log (\rho(x)) dx + \int V(x) \rho(x) dx \nonumber\\
	&\quad - c\iint \log|x - y|\rho(x) \rho(y) dx dy  + \log Z.
\end{align}
The functional $H_c(\mu)$ is strictly convex on any sublevel set, because $H(\mu | \alpha)$ is strictly convex, and the log-interaction term is convex (cf.~\cite[Lemma~1.8]{Saff-Totik-book}). Consequently, the minimizer of $H_c$ is unique and is absolutely continuous. We denote the minimizer and its density by $\mu_c$ and $\rho_c$, respectively.

%This is an easy consequence of the Donsker--Varadhan variational formula. Indeed, the variational formula implies that for a bounded measurable function $\psi \colon \R^2 \to \R$,
%\[
%	\iint \psi(x,y) d\mu(x) d\mu(y) \le 2 H(\mu | \alpha) + \log \iint e^{\psi(x,y)} d\alpha(x) d\alpha(y).
%\] 
%It follows that the inequality still holds for any non-negative measurable function by the monotone convergence theorem. Then using that formula for 
%\[
%\psi(x, y) = \begin{cases}
%\frac12 \big| \log|x-y| \big|, & x\neq y,\\
%0, & x = y, 
%\end{cases}
%\]
%we obtain that 
%\begin{align*}
%	&\frac12\iint  \big| \log|x-y| \big| d\mu(x) d\mu(y) \\
%	& \le 2 H(\mu | \alpha) + \log \iint e^{\frac12 |\log|x-y| |} d\alpha(x) d\alpha(y)\\
%	& = 2 H(\mu | \alpha) + \log \iint \left( |x-y|^{\frac12} \one_{ \{|x-y| \ge 1\}}  + |x - y|^{-\frac12} \one_{ \{|x-y| < 1\}}   \right)  d\alpha(x) d\alpha(y) < \infty.
%\end{align*}
%Here the integral in the last line is finite because the density $\alpha(x)$ is bounded and the first moment of $\alpha$ is finite. 

As a particular case of a general result in \cite{Liu-Wu-2019}, the sequence of empirical distributions $\{L_N\}$ of the beta ensemble~\eqref{VbE} 
%\begin{equation}\label{VbE}
%	(\lambda_1, \lambda_2, \dots, \lambda_N) \propto \frac{1}{Z_{\beta, N}} |\Delta(\lambda)|^{\beta} e^{- \sum_{i = 1}^N V(\lambda_i) },
%\end{equation}
in the regime where $ \beta  N\to 2 c \in (0, \infty)$ satisfies a LDP on $\cP(\R)$ with the good rate function $I_c$, where
\[
	I_c(\mu) = H_c(\mu) - \inf_{\nu \in \cP(\R)} H_c(\nu).
\]
It is worth noting that if the potential $V$ is assumed to be lower semi-continuous, then a LDP for $\{L_N\}$ holds under weaker conditions than the moments condition (see \cite[\S 4.3]{G-Zelada-2019}). Now the LDP implies the law of large numbers, that is, the sequence $\{L_N\}$ converges weakly to $\mu_c$, almost surely.

\begin{remark}
Consider the following beta ensembles 
	\[
		|\Delta (\lambda)|^\beta e^{-(1+\frac{\beta N}{2}) \sum_{i = 1}^N V(\lambda_i)}.
	\]
Then the limiting measure $\tilde \mu_c$ in the regime where $\beta N \to 2c \in [0, \infty]$ is the minimizer of the following functional 
\begin{align*}
	J_c(\rho) &= \int \rho(x) \log (\rho(x)) dx + (1+c)\int V(x) \rho(x) dx - c\iint \log|x - y|\rho(x) \rho(y) dx dy \\
	&=  H(\rho | \alpha) + c \E(\rho), \quad c \in [0, \infty),\\
	J_\infty &= \E(\mu).
\end{align*}
Under the assumption that $H(\mu_\infty | \alpha) < \infty$, we can show that $\{\tilde \mu_c\}$ is an interpolation between $\tilde \mu_0 = \alpha$ and $\tilde \mu_\infty = \mu_V$, that is, $\{\tilde \rho_c\}$ converges weakly to $\alpha$ (resp.~$\mu_V$) as $c \to 0$ (resp. $c \to \infty$).
\end{remark}

%The functional $H_c(\mu)$ is strictly convex on any sublevel set $\{\mu \in \cP(\R) : H(\mu | \alpha) \le L\}, (L>0)$. This is because the logarithmic energy
%\[
%	- \iint \log|x - y|\rho(x) \rho(y) dx dy
%\] 
%is also convex on that set (cf.~\cite[Lemma~1.8]{Saff-Totik-book}). Consequently, the minimizer of $H_c$, or of $I_c$, denoted by $\mu_c$, is unique. And thus, the LDP implies the almost sure convergence of $\{L_N\}$ to $\mu_c$. 

%We conclude this section by rewriting Theorem~\ref{thm:LLN-intro} from the introduction.
%\begin{theorem}[{cf.~\cite{G-Zelada-2019, Liu-Wu-2019}}]\label{thm:LLN-intro}
%In the regime where $\beta N \to 2c \in (0, \infty)$, the following hold.
%	\begin{itemize}
%		\item[\rm (i)] The sequence of empirical distributions $L_N$ satisfies the LDP with the good rate function $I_c$. 
%		\item[\rm (ii)] The function $I_c$ is strictly convex and has a unique minimizer $\mu_c$ which is absolutely continuous with density $\rho_c$.
%		\item[\rm (iii)] The sequence $\{L_N\}$ converges weakly to $\mu_c$, almost surely.
%	\end{itemize}
%\end{theorem}

\section{Equilibrium measures}
The approach introduced in the previous section shows that in the regime where $\beta N \to 2c \in (0, \infty)$, the empirical distribution $L_N$ converges weakly to the limiting measure $\mu_c$ which is the minimizer of the energy functional $H_c$, almost surely. 

In this section, we are going to derive an equation characterizing $\rho_c$ in a rigorous way.
Let us first explain some main ideas. Let $\rho_{\beta, N}(x)$ (or $\rho_N(x)$ for short) be the first marginal of the beta ensemble~\eqref{VbE}
\begin{align}
	\rho_N(x) &= \frac{1}{Z_{\beta, N}} e^{-V(x)}\idotsint \left( \prod_{i = 1}^{N-1} |x - \lambda_i|^\beta \right)|\Delta(\lambda)|^\beta e^{-\sum_{i = 1}^{N-1} V(\lambda_i)} d\lambda_1 \cdots d\lambda_{N - 1} \nonumber\\
	&= \frac{Z_{\beta, N - 1}}{Z_{\beta,N}} e^{-V(x)} \Ex_{\beta, N-1} \left[ e^{\beta \sum_{j = 1}^{N-1} \log |x - \lambda_j|} \right].  \label{1-point}
\end{align}
Here $ \Ex_{\beta, N-1}[\cdot] $ denotes the expectation with respect to the beta ensemble~\eqref{VbE} of $(N-1)$ particles. Since the measure $\rho_N(x)dx$ is the mean measure of the random probability measure $L_N$, that is, for integrable function $f \colon \R \to \R$, 
\[
	\Ex[\bra{L_N, f}] = \int f(x) \rho_N(x) dx,
\]
it follows that it also converges weakly to $\rho_c(x)dx$ in the considering regime. Here the notation $\bra{\mu, f}$ denotes the integral $\int f d\mu$, for a measure $\mu$ and an integrable function $f$. Note that by definition, the ratio ${Z_{\beta, N}}/{Z_{\beta, N - 1}} $ can be written as 
\begin{equation}\label{ratio}
	\frac{Z_{\beta, N}}{Z_{\beta, N - 1}}  = \int \Ex_{\beta, N-1} \left[ e^{\beta \sum_{j = 1}^{N-1} \log |x - \lambda_j|} \right] e^{-V(x)} dx.
\end{equation}
To characterize the limiting measure $\rho_c(x) dx$, we are going to study the limit of the expectation inside the above integrand.

It is clear that the almost sure convergence in Theorem~\ref{thm:LLN-intro} still holds when $N$ is replaced by $(N-1)$. And thus, although the function $y \mapsto \log|x - y|$ is neither bounded nor continuous, one may expect that almost surely,
\begin{equation}\label{log-function}
	\beta \sum_{j = 1}^{N-1} \log |x - \lambda_j| = \beta (N-1) \int \log|x - y| dL_{N-1} (y)  \to 2 c  \int \log|x-y| \rho_c(y)dy.
\end{equation}
We will show that the above convergence holds in probability instead, by a truncation method with the help of some estimates. 

Lemma~\ref{lem:moment} provides a crucial estimate that 
\begin{align*}
	\Ex_{\beta, N-1} \left[ e^{\beta \sum_{j = 1}^{N-1} \log |x - \lambda_j|} \right] \le M (1 + x^2)^{\frac \kappa 2},
\end{align*}
whenever $\beta (N-1)  \le \kappa$,
where $M$ is a constant depending only on $\kappa$. To show this, we use the decoupling inequality in \cite{Decoupling} as an important tool. A lower bound for the ratio ${Z_{\beta, N}}/{Z_{\beta, N - 1}} $ can be easily deduced by using Jensen's inequality. Consequently, the density $\rho_N(x)$ is uniformly bounded by 
\begin{equation}\label{first-point}
	\rho_N(x) \le \Lambda  (1 + x^2)^{\frac \kappa2} e^{-V(x)}, \quad (\beta N \le \kappa),
\end{equation}
for a constant $\Lambda$ depending on $\kappa$. This is where we need the moments assumption.
That estimate enables us to show the convergence of $\{\bra{L_N, f}\}$ for continuous function of polynomial growth, and to handle the singularity of the log function at $y = x$.

Once the convergence~\eqref{log-function} in probability is proved, by the continuous mapping theorem, it follows that 
\[
	e^{\beta \sum_{j = 1}^{N-1} \log |x - \lambda_j|} \to e^{2 c  \int \log|x-y|  \rho_c(y)dy} \quad \text{in probability.}
\] 
A generalization of the estimate~\eqref{first-point} in Lemma~\ref{lem:k-bound} implies that the above sequence is uniformly integrable. Thus, the convergence of the expectations follows. Then by letting $N \to \infty$ in the equations~\eqref{ratio} and~\eqref{1-point}, we get that,  
\begin{align*}
	&\frac{Z_{\beta, N}}{Z_{\beta, N - 1}} \to \int  e^{2 c  \int \log|x-y|  \rho_c(y)dy} e^{-V(x)} dx =: Z_c,\\
	& \rho_N(x) \to \frac{1}{Z_c} e^{-V(x) + 2c \int \log|x-y| \rho_c(y) dy}.
\end{align*}
Recall that $\{\rho_N(x)dx\}$ converges weakly to $\rho_c(x) dx$. From those, the equation characterizing $\rho_c$ in Theorem~\ref{main-result}(i) is derived.

\subsection{Some initial estimates}

The following general estimate may have its own interest and will be used to bound the expectation $\Ex[e^{\beta \sum_{i=1}^N \log|x - \lambda_i| }]$.

\begin{proposition}\label{lem:exp}
Assume that $\varphi \colon \R \to \R$ is a measurable function satisfying 
\[
	\int_\R |x|^k e^{\varphi(x) - V(x)} dx < \infty, \quad k = 0,1,2, \dots.
\]
Then for $\beta N \le \kappa$,
\[
	\Ex \left[e^{\frac1N \sum_{i = 1}^N \varphi(\lambda_i)} \right]  \le \left(\Ex \left[e^{ \sum_{i = 1}^N \varphi(\lambda_i)} \right] \right)^{1/N} \le M_{\varphi, \kappa},
\]
where $M_{\varphi, \kappa}$ is a constant depending on $\varphi$ and $\kappa$.
\end{proposition}

\begin{proof}
Since  
\begin{align*}
	\Ex \left[e^{ \sum_{i = 1}^N \varphi(\lambda_i)} \right] = \frac{1}{Z_{\beta, N}} \idotsint |\Delta(\lambda)|^\beta \prod_{i = 1}^N e^{(\varphi - V) (\lambda_i)} d\lambda_i =: \frac{Z_N^{(\varphi)}}{Z_{\beta, N}},
\end{align*}
the desired result follows from a lower bound for $Z_{\beta, N}$ and an upper bound for $Z_N^{(\varphi)}$ which are proved in 
Lemma~\ref{lem:partition-lb} and Lemma~\ref{lem:exp-varphi}, respectively.
\end{proof}

For the normalizing constant $Z_{\beta, N}$,  also called a partition function 
\[
	Z_{\beta, N} =   \idotsint |\Delta(\lambda)|^\beta e^{-\sum_{i = 1}^N V(\lambda_i)} d\lambda_1 \cdots d\lambda_N,
\]
a lower bound is an easy consequence of Jensen's inequality.

\begin{lemma}\label{lem:partition-lb}
For $\kappa > 0$, there  is a constant $C_\kappa \in \R$ such that for $\beta N \le \kappa$,
	\[
		\frac{\log Z_{\beta, N} }{N}  \ge C_\kappa .
	\]
\end{lemma}
\begin{proof}
	Let us express $Z_{\beta, N}$ in terms of the integral with respect to the probability measure $\alpha$,
\[
	Z_{\beta, N} = Z^N \idotsint e^{\frac\beta 2\sum_{i \neq j} \log|\lambda_j - \lambda_i| } d\alpha(\lambda_1) \cdots d\alpha(\lambda_N).
\]
Then by Jensen's inequality, we obtain that
	\begin{align*}
		\log Z_{\beta,N} &\ge N \log Z + \frac{\beta N(N-1)}{2} \iint \log|x - y| d\alpha(x) d\alpha(y) \\
		&\ge N\left(\log Z + \begin{cases}
			 0, &\text{if } \ell_\alpha = \iint \log|x - y| d\alpha(x) d\alpha(y) \ge 0\\
			\frac{\kappa \ell_\alpha}2, &\text{otherwise}
		\end{cases} \right).
	\end{align*}
The proof is complete.
\end{proof}

Next, we study $Z_N^{(\varphi)}$.

\begin{lemma}\label{lem:exp-varphi}
%	Assume that $\varphi \colon \R \to \R$ is a measurable function satisfying 
%\[
%	\int_\R |x|^k e^{\varphi(x) - V(x)} dx < \infty, \quad k = 0,1,2, \dots.
%\]
For $\beta N \le \kappa$,
\[
	\frac{\log Z_N^{(\varphi)}}N \le C_{\varphi, \kappa},
\]
where $C_{\varphi, \kappa}$ is a constant depending on $\varphi$ and $\kappa$. 
%Here recall that 
%\[
%	Z_N^{(\varphi)} = \idotsint |\Delta(\lambda)|^\beta \prod_{i = 1}^N e^{(\varphi - V) (\lambda_i)} d\lambda_i.
%\]

\end{lemma}
To prove this lemma, we need the following inequalities which are special cases of Theorem~2.1 in \cite{Decoupling} (decoupling inequality) and Lemma~3.4 in \cite{Liu-Wu-2019}, respectively.

\begin{lemma}
Let $\{X_i\}_{i = 1}^N$ be an i.i.d.~(independent identically distributed) sequence of random variables on $\R$, and let $\{Y_i\}_{i = 1}^N$ be its independent copy. Let $\psi$ be a convex increasing function on $[0, \infty)$, and $\Phi \colon \R^2 \to \R$ be a symmetric function such that $\Ex[|\Phi(X_1, X_2)|] < \infty$. Then the following inequalities hold
\begin{equation}
	\Ex \bigg[ \psi \bigg( \bigg| \sum_{i \neq j} \Phi(X_i, X_j) \bigg|\bigg) \bigg] \le \Ex \bigg[ \psi \bigg( 8 \bigg| \sum_{i \neq j} \Phi(X_i, Y_j) \bigg|\bigg) \bigg],
\end{equation}
\begin{equation}
	\log \Ex \bigg[ \exp\bigg( \frac{1}{N(N-1)} \sum_{i \neq j} \Phi(X_i, Y_j) \bigg) \bigg] \le (N-1) \log \Ex \bigg[ \exp\bigg( \frac{1}{N-1} \Phi(X_1, Y_1) \bigg)\bigg].
\end{equation}
\end{lemma} 
%
%\begin{lemma}[\cite[Lemma~3.4]{Liu-Wu-2019}]
%
%\end{lemma}
\begin{proof}[Proof of Lemma~{\rm\ref{lem:exp-varphi}}]
	Let $Z_\varphi$ be the normalizing constant of the probability measure 
	\[
		d\mu(x) = \frac1{ Z_\varphi} e^{\varphi(x) - V(x)}dx.
	\]
Let $\Phi(x, y) = \log|x - y| \vee 0 = \max\{\log|x - y|, 0\}$. Then it is clear that
	\[
		Z_N^{(\varphi)} =(Z_\varphi)^N \Ex\bigg[\exp \bigg(\frac \beta 2 \sum_{i \neq j} \log|X_i - X_j| \bigg)\bigg] \le  (Z_\varphi)^N \Ex\bigg[\exp \bigg(\frac \beta 2 \sum_{i \neq j} \Phi(X_i, X_j) \bigg)\bigg],
	\]
where $\{X_i\}_{i = 1}^N$ is an i.i.d.~sequence of random variables with common distribution $\mu$. Let $\{Y_i\}_{i = 1}^N$ be an independent copy of $\{X_i\}_{i = 1}^N$. Then using the two inequalities (with $\psi = \exp$ in the decoupling inequality) quoted in the above lemma consecutively, we deduce that 
\begin{align*}
	\Ex\bigg[\exp \bigg(\frac \beta 2 \sum_{i \neq j} \Phi(X_i, X_j) \bigg)\bigg] &{\le} \Ex\bigg[\exp \bigg(4 \beta \sum_{i \neq j} \Phi(X_i, Y_j) \bigg)\bigg] \\
	& \le \exp \left( (N - 1) \log \Ex \left[e^{4 \beta N \Phi(X_1, Y_1)} \right]\right) \\
	&\le \exp \left( N  \log \Ex \left[e^{4 \kappa \Phi(X_1, Y_1)} \right]\right)\\
	&\le \exp \left( N  \log \Ex \left[ 1 + |X_1-Y_1|^{4 \kappa} \right]\right).
\end{align*}
Finally, the upper bound is obtained by taking the logarithm
\begin{align*}
	\frac{\log Z_N^{(\varphi)} }{N} & \le  \log Z_\varphi  + \log \left(1 + \Ex[|X_1 - Y_1|^{4 \kappa}] \right) =: C_{\varphi, \kappa},
\end{align*}
which completes the proof.
\end{proof}

Using Proposition~\ref{lem:exp}, we now bound the expectation 
\[
	\Ex \bigg[\prod_{i = 1}^N |x - \lambda_i|^\beta \bigg] = \Ex \left[e^{\beta \sum_{i=1}^N \log|x - \lambda_i| } \right],
\]
an important step to bound the density $\rho_N(x)$.

\begin{lemma}\label{lem:moment}
For $\kappa > 0$, there is a constant $M = M (\kappa)$ such that for $\beta N \le \kappa$,
	\begin{equation}\label{upper-bound-for-Z}
		\Ex \bigg[\prod_{i = 1}^N |x - \lambda_i|^\beta \bigg] \le M (1 + x^2)^{\frac{\kappa}{2}}.
	\end{equation}
\end{lemma}
\begin{proof}
It follows from the inequality 
\[
	|x - \lambda|^2 \le (1 + x^2) (1+ \lambda^2),
\]
that
\begin{align*}
	 \beta  \sum_{i = 1}^N \log|x - \lambda_i|  &\le \frac{ \beta N }{2}\log(1+x^2) + \frac1N \sum_{i = 1}^N  \frac{\beta N}2 \log(1 + \lambda_i^2) \\
	 &\le \frac{ \kappa }{2}\log(1+x^2) + \frac1N \sum_{i = 1}^N  \frac\kappa2 \log(1 + \lambda_i^2).
\end{align*}
Therefore, 
\[
	\Ex\bigg[ \prod_{i = 1}^N |x - \lambda_i|^\beta\bigg] \le \Ex \left[e^{\frac1N \sum_{i = 1}^N \frac  \kappa 2 \log(1 + \lambda_i^2)} \right] \times (1+x^2)^{\frac \kappa 2}.
\]
The expectation on the right hand side is bounded by a constant $M=M(\kappa)$ by using Proposition~\ref{lem:exp} for the function $\varphi (\lambda) = \frac\kappa 2 \log (1 + \lambda^2 )$. The proof is complete.
\end{proof}

\begin{proposition}\label{prop:Wegner}
For $\kappa > 0$, there is a constant $\Lambda = \Lambda (\kappa)$ such that for $\beta N \le \kappa$,
	\begin{equation}\label{Wegner}
		\rho_N(x) \le \Lambda (1+x^2)^{\frac \kappa 2} e^{-V(x)}.
	\end{equation}
In particular, together with the assumption on the potential $V$, it follows that in the regime $\beta N \to 2c$, there is a constant $D >0$ such that 
\[
	\rho_N(x) \le D, \quad \text{for any $x \in \R$, and any $N$.}
\]

\end{proposition}
\begin{proof}
It follows from the expression of the function $\rho_N(x)$ in~\eqref{1-point} and the estimate~\eqref{upper-bound-for-Z} that
\[
	\rho_N(x) \le \frac{Z_{\beta, N-1}}{Z_{\beta, N}} M (1+x^2)^{\frac \kappa 2} e^{-V(x)}.
\]
The ratio $Z_{\beta, N-1}/Z_{\beta, N}$ is bounded by a constant $C$, which will be proved in Lemma~\ref{lem:fraction-bound} below. Then the estimate~\eqref{Wegner} holds with $\Lambda = M C$. The proof of this proposition is complete.
\end{proof}

The following estimate is analogous to Lemma~4.4 in \cite{Johansson-1998}.
\begin{lemma}\label{lem:fraction-bound}
For $\kappa>0$, there is a constant $C = C(\kappa)>0$ such that for $\beta N \le \kappa$, 
\[
	\frac{Z_{\beta, N - 1}}{Z_{\beta, N}} \le C.
\]
\end{lemma}
\begin{proof}
Jensen's inequality implies that
\[
	\Ex_{\beta, N-1}\left[ e^{\beta \sum_{i = 1}^{N-1} \log|x - \lambda_i|} \right] \ge e^{\beta (N-1) \int \log|x - y|  \rho_{\beta, N -1} (y) dy}.
\]
Here $\rho_{\beta, N-1}$ is the first marginal of the ensemble~\eqref{VbE} with parameters $(\beta, N-1)$.
Using Jensen's inequality again, we deduce that 
\begin{align*}
	\frac{Z_{\beta, N}}{Z_{\beta, N - 1}} &= \int   e^{-V(x)} \Ex_{\beta, N-1} \left[ e^{\beta \sum_{i = 1}^{N-1} \log|x - \lambda_i|} \right] dx\\
	&\ge Z \int e^{\beta (N-1) \int \log|x - y|  \rho_{\beta, N -1} (y) dy } d\alpha (x) \\
	&\ge Z e^{ \beta(N-1)  \iint \log|x - y|   \rho_{\beta, N-1} (y) dy d\alpha(x)}.
\end{align*}

Let $\norm{\alpha}_\infty = Z^{-1} \sup_{x} e^{-V(x)} < \infty$. Then it is clear that 
\begin{align*}
	\iint \log|x-y|   \rho_{\beta, N -1}(y) dy d\alpha (x) &= \int \left( \int \log|x - y| d\alpha(x) \right)  \rho_{\beta, N - 1}(y) dy \\
	&\ge \int \left( \int_{x : |x - y|\le 1} \log|x - y| d\alpha(x)\right)  \rho_{\beta, N - 1}(y) dy \\
	&\ge \int \left( \int_{x : |x - y|\le 1} \log|x - y| \norm{\alpha}_\infty dx \right)  \rho_{\beta, N - 1}(y) dy \\
	&= -2\norm{\alpha}_\infty.
\end{align*}
Since $\beta(N-1) < \beta N \le \kappa$, we conclude that  
\[
	\frac{Z_{\beta, N}}{Z_{\beta, N - 1}} \ge Z e^{- 2\kappa  \norm{\alpha}_\infty},\quad \text{or}\quad 		\frac{Z_{\beta, N-1}}{Z_{\beta, N}} \le Z^{-1} e^{ 2\kappa  \norm{\alpha}_\infty} =:C.
\]
The proof is complete.
\end{proof}

The next lemma which is a generalization of Lemma~\ref{lem:moment}  helps us to show the uniform integrability of $\{\prod_{i=1}^N |x-\lambda_i|^\beta\}$ and also implies a bounded condition, one of the two sufficient conditions for the local Poisson statistics (see Section~4).
\begin{lemma}\label{lem:k-bound}
	Let $\kappa > 0$ and $L > 0$. Then there is a number $\theta = \theta(\kappa, L)$ such that for $\beta k \vee \beta N \le \kappa$, and any $x_1, x_2, \dots, x_k \in [-L, L]$,
	\begin{align*}
		\Ex\left[ \prod_{j = 1}^k \prod_{i = 1}^N |x_j - \lambda_i|^\beta\right] \le \theta^k.
	\end{align*}
\end{lemma}
\begin{proof}
	The proof is similar to that of Lemma~\ref{lem:moment}. We begin with the following inequality 
	\begin{align*}
		\beta  \sum_{j = 1}^k \sum_{i = 1}^N \log|x_j - \lambda_i|  &\le  \frac\beta2  \sum_{j = 1}^k \sum_{i = 1}^N \Big( \log(1 + x_j^2) + \log (1 + \lambda_i^2) 	\Big)\\
		&\le \frac{k \beta N }{2}\log(1+L^2) + \frac {k \beta}{2} \sum_{i = 1}^N  \log(1 + \lambda_i^2) \\
		&\le  \frac{k \kappa }{2}\log(1+L^2) + \begin{cases}
		\frac kN \sum_{i = 1}^N \frac{ \kappa}{2} \log(1 + \lambda_i^2), &\text{if $k \le N$,}\\
		 \sum_{i = 1}^N \frac{ \kappa}{2} \log(1 + \lambda_i^2), &\text{if $k > N$.}
		 \end{cases}
	\end{align*}
It follows that
\begin{align*}
	\Ex\left[ \prod_{j = 1}^k \prod_{i = 1}^N |x_j - \lambda_i|^\beta\right] &\le (1+L^2)^{\frac {k \kappa} 2}
	\times 
	\begin{cases}
		 \Ex\left[ e^{\frac kN \sum_{i = 1}^N \frac{ \kappa}{2} \log(1 + \lambda_i^2)}\right], &\text{if $k \le N$,}\\
		  \Ex\left[ e^{\sum_{i = 1}^N \frac{ \kappa}{2} \log(1 + \lambda_i^2)}\right], &\text{if $k > N$,}
	\end{cases}\\
	&\le  (1+L^2)^{\frac {k \kappa} 2}  \times\left( \Ex\left[ e^{ \sum_{i = 1}^N \frac{  \kappa}{2} \log(1 + \lambda_i^2)}\right]  ^{\frac{1}N}  \right)^k \\
	&\le \theta ^k,
\end{align*}
for some constant $\theta = \theta(\kappa, L) > 0$. Here in case $k \le N$, H\"older's inequality has been used. The proof is complete.
\end{proof}

\begin{corollary}
	For $x \in \R$, let 
	\[
		X_N(x) = \prod_{i = 1}^N |x - \lambda_i|^\beta.
	\]
Then for any compact set $K \subset \R$,
\[
	\sup_{\substack{\beta N \le \kappa, x \in K\\}} \Ex[(X_N(x))^2] < \infty.
\]
In particular, for any bounded sequence $\{x_N\} \subset \R$, in the regime where $\beta N \to 2c \in (0, \infty)$, the sequence $\{X_N(x_N)\}$ is uniformly integrable.
\end{corollary}

\subsection{Continuous functions of polynomial growth}

We consider here a weaker type of convergence, convergence in probability. 
We first show that the sequence of moments of the empirical distribution $L_N$ also converges. Note that the weak convergence of probability measures does not imply the convergence of moments.
%
%\begin{lemma}\label{lem:moment-bounded}
%For any  $\kappa > 0$ and $k > 0$,
%\begin{equation}\label{moment}
%	\sup_{ \beta N \le \kappa} \Ex[\bra{L_N, |x|^k}]  < \infty. 
%\end{equation}
%
%
%
%\end{lemma}
%\begin{proof}
%This is a direct consequence of the upper bound for the first point function $\rho_N(x)$ in Proposition~\ref{prop:Wegner}. Indeed, recall from that upper bound that
%\[
%	\rho_N(x) \le \Lambda (1+x^2)^{\frac\kappa 2} e^{-V(x)}.
%\]
%Therefore, for $\beta N \le \kappa$, 
%\[
%	\Ex[\bra{L_N, |x|^k}] = \int |x|^k \rho_N(x) dx \le \Lambda \int |x|^k (1+x^2)^{\frac\kappa 2} e^{-V(x)}dx < \infty,
%\]
%which completes the proof.
%\end{proof}

\begin{lemma}
For any $k \in \N$, as $N \to \infty$ with $\beta N \to 2c \in (0, \infty)$,
\[
	\bra{L_N, x^k} \to \bra{\mu_c, x^k} \quad \text{ in $L^1$, and in probability.}
\] 
\end{lemma}
\begin{proof}
In the regime where $\beta N \to 2c \in (0, \infty)$, we claim that for $k \in \N$, 
\[
	M_{2k} = \sup_{N} \Ex[\bra{L_N, x^{2k}}] < \infty.	
\] 
Indeed, let $\kappa = \sup_{N} (\beta N)$, and let $\Lambda=\Lambda(\kappa)$ be the constant in Proposition~\ref{prop:Wegner}. Then for any $N$, 
\[
	 \Ex[\bra{L_N, x^{2k}}] = \int x^{2k} \rho_N(x) dx \le \Lambda \int x^{2k} (1+x^2)^{\frac\kappa 2} e^{-V(x)}dx =: M_{2k} < \infty,
\]
which proves the claim.

	For $L > 0$, let 
\[
	f_L(x) = ( x^k \wedge L)\vee (-L) =\begin{cases}
	 x^k ,&\text{if $-L \le x^k \le L$,}\\
	 -L ,&\text{if $x^k \le -L$,}\\
	 L ,&\text{if $x^k \ge L$.}
	 \end{cases}
\]
Since $f_L(x)$ is a bounded continuous function, as $N \to \infty$,
	\[
		\bra{L_N, f_L} \to \bra{\mu_c, f_L} \quad \text{in $L^1$.}
	\] 
It then follows that $\bra{\mu_c, x^{k}} = \lim_{L \to \infty} \bra{\mu_c, f_L} \le M_k$, for any even $k$. Consequently, all moments of $\mu_c$ are finite.
	
	It is clear that 
	\[
		|\bra{L_N, x^k} - \bra{L_N, f_L}| \le \bra{L_N, |x|^k \one_{\{|x|^k > L\}}} \le \frac{1}{L} \bra{L_N, x^{2k}}.
	\]
Then, by taking the expectations of both sides, we obtain that  
\begin{equation}\label{fL}
	0 \le \Ex[|\bra{L_N, x^k} - \bra{L_N, f_L}|] \le \frac{ \Ex[\bra{L_N, x^{2k}}] }{L} \le  \frac{M_{2k}}{L}.
\end{equation}
Next, the triangular inequality yields
\begin{align*}
	\Ex[|\bra{L_N, x^k} - \bra{\mu_c, x^k}|] &\le \Ex[|\bra{L_N, x^k} -  \bra{L_N, f_L}|] +\Ex[| \bra{L_N, f_L} - \bra{\mu_c, f_L}|] \\
	&\quad + |\bra{\mu_c, x^k - f_L}|.
\end{align*}
The first term and the third term are bounded by $M_{2k} /{L}$ by \eqref{fL}. The second term converges to zero as $N \to \infty$. These imply that $\Ex[|\bra{L_N, x^k} - \bra{\mu_c, x^k}|] \to 0$ as $N \to \infty$. The proof is complete.
\end{proof}

The following result is a consequence of the convergence of moments.
\begin{theorem}
	Let $f$ be a continuous function of polynomial growth. Then as $\beta N \to 2c \in (0, \infty)$,
	\[
		\bra{L_N, f} \to \bra{\mu_c, f} \quad \text{in probability.}
	\]
\end{theorem}
\begin{proof}
It is clear that the convergence holds, if $f$ is a polynomial. Now assume that the function $f$ is continuous of polynomial growth, meaning that there is a polynomial $p$ such that $|f(x)| \le p(x)$ for all $x \in \R$. Then by a similar truncation argument as used in the proof of the above lemma, we can easily deduce that the sequence $\{\bra{L_N, f}\}$ converges in probability to $\bra{\mu_c, f}$. 
%One may also see the proofs of Lemma~2.1 and Lemma~2.2 in \cite{Trinh-ojm-2018} for  detailed arguments. 
The theorem is proved.
\end{proof}
%
%A consequence of the convergence of moments is the convergence for continuous test functions of polynomial growth. For the proof of the derivation, see Lemma 2.2 in \cite{Trinh-ojm-2018}, for example.

\subsection{$\log$ functions}

\begin{lemma}\label{lem:log}
Let $x \in \R$ be given. Then as $\beta N \to 2c$,
	\[
		\frac1N \sum_{i = 1}^{N} \log|x - \lambda_i|  \to \int \log|x - y| \rho_c(y) dy \quad \text{in probability.}
	\]
%Recall that $\rho_c$ is the density of the limiting measure $\mu_c$.
\end{lemma}
\begin{proof}
	We first remark that since $\rho_N(x) \le D$, and since the sequence of probability measures $\{\rho_N(x) dx\}$ converges weakly to $\rho_c(x)dx$, it follows that the density $\rho_c$ is also bounded by $D$ (almost everywhere). In addition, recall that all moments of $\rho_c$ are finite. Thus, for any $x\in \R$,
	\[
		\int \log|x - y| \rho_c(y) dy < \infty.
	\]
	
For $L>0$, the truncation $f_L(y)=\log|x-y| \vee (-L)$ is a continuous function of polynomial growth, and hence, as $N \to \infty$,
\[
	\frac{1}{N} \sum_{i = 1}^N \log|x - \lambda_i| \vee (-L) \to \int (\log|x - y| \vee (-L)) \rho_c(y) dy \quad \text{in probability. }
\]
Next, we use the inequality 
\[	
	|\log|x-y| - f_L(y)| \le -\log|x-y| \one_{\{|y-x| \le e^{-L}\}}
\]
to deduce that 
\begin{align*}
	\Ex[|\bra{L_N, \log|x-\cdot|} - \bra{L_N, f_L}|] &\le \Ex[\bra{L_N, -\log|x-\cdot| \one_{\{|\cdot - x| \le e^{-L} \}}}] \\
	&= \int_{[x - e^{-L}, x+e^{-L}]} (-\log|x - y|) \rho_N(y)dy \\
	&\le M \int_{[x - e^{-L}, x+e^{-L}]} (-\log|x - y|)dy \\
	&=2M (1+L) e^{-L}\\
	&\to 0 \quad \text{as}\quad L \to \infty.
\end{align*}
The desired result follows easily from the triangular inequality. The proof is complete.
\end{proof}

\subsection{Partition functions}
Let $\Prob_{\beta, N-1}$ denote the probability of the beta ensemble~\eqref{VbE} with parameters $\beta$ and $(N-1)$. Then  Lemma~\ref{lem:log} still holds if $N$ is replaced by $(N-1)$, that is, under $\Prob_{\beta, N-1}$, for fixed $x \in \R$, as $\beta N \to 2c$,
\[
	\frac{1}{N-1} \sum_{i = 1}^{N - 1} \log|x - \lambda_i| \to \int \log|x-y| \rho_c(y)dy \quad \text{in probability}.
\]
From which, we get the following results.
\begin{lemma}\label{lem:Ex-of-Z}
For fixed $x \in \R$, as $\beta N \to 2c \in (0, \infty)$,
	\[
		 \prod_{i = 1}^{N-1}|x - \lambda_i|^\beta = e^{\beta \sum_{i = 1}^{N-1} \log|x - \lambda_i|} \to e^{2c \int \log|x - y| \rho_c(y) dy}
	\]
in probability under $\Prob_{\beta, N -1}$ by the continuous mapping theorem, and then
\begin{equation}\label{convergence-of-Z}
	\Ex_{\beta, N - 1}\bigg[\prod_{i = 1}^{N-1}|x - \lambda_i|^\beta \bigg] \to e^{2c \int \log|x - y| \rho_c(y) dy},
\end{equation}
by the uniform integrability.
\end{lemma}

\begin{lemma}\label{lem:fraction}
As $\beta N \to 2c$,
	\begin{align*}
		\frac{Z_{\beta, N}}{Z_{\beta, N - 1}} &=  \int   e^{-V(x)} \Ex_{\beta, N - 1}\bigg[\prod_{i = 1}^{N-1}|x - \lambda_i|^\beta \bigg] dx \\
		&\to \int e^{-V(x) + 2c \int \log|x - y| \rho_c(y) dy} dx.
	\end{align*}
\end{lemma}
\begin{proof}
The desired result follows from \eqref{upper-bound-for-Z} and \eqref{convergence-of-Z} by using Lebesgue's dominated convergence theorem.
\end{proof}

\subsection{Proof of Theorem~\ref{main-result}(i)}
In this section, we give a proof of the first part of our main result (Theorem~\ref{main-result}(i)) which is restated here for convenience.
\begin{theorem}\label{thm:global}
	The limiting measure $\mu_c$ in the regime where $\beta N \to 2c \in (0, \infty)$ has bounded density $\rho_c$ which can be chosen to satisfy the following equation 
\begin{equation}\label{rhoc}
	\rho_c(x) = \frac{1}{Z_{c} }e^{-V(x) + 2 c  \int \log|x - y| \rho_c(y) dy}, \quad \text{for all $x \in \R$.}
\end{equation}
In particular, $\rho_c(x) > 0$, for all $x \in \R$, and thus, the measure $\mu_c$ has full support.
\end{theorem}

\begin{proof}
Recall that  the first marginal $\rho_N(x)$ can be expressed as
\begin{align*}
	\rho_N(x) 
		&=\frac{Z_{\beta, N - 1}}{Z_{\beta, N}} e^{-V(x)} \Ex_{\beta, N-1} \bigg[\prod_{i = 1}^{N - 1}{|x - \lambda_i|^\beta} \bigg].
\end{align*}
Then it follows from Lemma~\ref{lem:Ex-of-Z} and Lemma~\ref{lem:fraction} that as $N \to \infty$,
\[
	\rho_N(x) \to \frac{e^{-V(x) + 2 c  \int \log|x - y| \rho_c(y) dy }}{ \int e^{-V(t) + 2c \int \log|t - y| \rho_c(y) dy} dt} = \frac{1}{Z_{c} }e^{-V(x) + 2 c  \int \log|x - y| \rho_c(y) dy} =: \tilde\rho_c(x).
\]

Recall also that $\rho_N(x) \le D$, for any $x \in \R$ and any $N$. Thus, for any $-\infty <a< b  < \infty$,
\[
	\int_{a}^b \rho_N(x) dx \to \int_a^b \tilde\rho_c(x) dx,
\]
by the bounded convergence theorem. On the other hand, since the sequence of measures $\{\rho_N(x)dx\}$ converges weakly to $\rho_c(x) dx$, it follows that
\[
	\int_{a}^b \rho_N(x) dx \to \int_a^b \rho_c(x)dx.
\]
Therefore, $\rho_c(x) = \tilde\rho_c(x)$, for almost every $x \in \R$. Modify the density $\rho_c(x)$ by taking $\rho_c(x) = \tilde\rho_c(x)$ for all $x \in \R$, we get the relation
\begin{align*}
	\rho_c(x) = \frac{1}{Z_{c} }e^{-V(x) + 2 c  \int \log|x - y| \rho_c(y) dy}, \quad \text{for all $x \in \R$.}
\end{align*}
This implies that $\rho_c(x) > 0$ for all $x \in \R$, meaning that the limiting measure $\mu_c$ has full support. The proof is complete. 
\end{proof}

\section{Poisson statistics}
The aim of this section is to prove Theorem~\ref{main-result}(ii) which is also restated here for convenience.
\begin{theorem}\label{thm:local}
Assume that the potential $V$ is continuous and that 
\[
	\lim_{x \to \pm \infty} \frac{V(x)}{\log(1+x^2)} = \infty.
\]
Then for any fixed $E \in\R$, the local statistics $\xi_N(E)$ converges weakly to a homogeneous Poisson point process on $\R$ with density $\rho_c(E)>0$. Here note that the density $\rho_c(x)$ chosen to satisfy the relation in Theorem~{\rm\ref{thm:global}} is continuous.
\end{theorem}

The ideas of proving the local Poisson statistics are as follows. 
Let $R_N^{(k)}$ denote the $k$th correlation function  of $\xi_N(E)$ which is given by 
\[
	R_N^{(k)} (x_1, x_2, \dots, x_k) = \frac{N!}{N^k(N-k)!}\rho_N^{(k)} \left(E+ \frac{x_1}N, E+\frac{x_2}N, \dots, E + \frac{x_k}N \right),
\] 
where $\rho_N^{(k)}$ is the $k$-dimensional marginal of the beta ensembles~\eqref{VbE}
\begin{align}
	\rho_N^{(k)}(x_1, \dots, x_k) &= \frac{1}{Z_{\beta, N}}|\Delta(x)|^\beta e^{-\sum_{j = 1}^k V(x_j)}  \nonumber\\
	& \times \idotsint  \bigg( \prod_{\substack{1 \le j \le k,\\ 1 \le i \le N-k}}  |x_j - \lambda_i|^\beta \bigg) |\Delta(\lambda)|^\beta e^{-\sum_{i = 1}^{N - k} V(\lambda_i)} d\lambda_1 \cdots d\lambda_{N - k} \nonumber\\
	&=\frac{Z_{\beta, N-k}}{Z_{\beta, N}} |\Delta(x)|^\beta e^{-\sum_{j = 1}^k V(x_j)} \Ex_{\beta, N-k}\bigg[ \prod_{j = 1}^k \prod_{i = 1}^{N - k} |x_j - \lambda_i|^\beta\bigg]. \label{k-marginal}
\end{align}
%with $Z_{\beta, N-k}$ the normalizing constant of the ensemble for $(N - k)$ particles. 
Here the points $\{x_1, \dots, x_k\}$ are assumed to be distinct. In order to show that the local statistics $\xi_N(E)$ converges to a homogeneous Poisson process with density $\rho_c(E)$, it suffices to prove the following two conditions (see Appendix in \cite{Peche-2015})
\begin{itemize}
	\item[(A)] for distinct $\{x_1, x_2, \dots, x_k\}$,
	\[
		R_N^{(k)} (x_1, x_2, \dots, x_k) \to \rho_c(E)^k;
	\]
	\item[(B)] and for any compact set $K \subset \R$, there is a constant $\theta_K$ such that 
\begin{equation*}
	R_N^{(k)} (x_1, x_2, \dots, x_k) \le  ( \theta_K)^k, \quad \text{for $k \le N$, for any $x_i \in K$.}
\end{equation*}
\end{itemize}
The latter is a consequence of Lemma~\ref{lem:k-bound} while the former follows from the continuity of $V$ and the following result which is an extension of Lemma~\ref{lem:log}.

\begin{lemma}\label{lem:log-E}
	Let $E \in \R$ and $x \in \R$ be given. Then as $N \to \infty$ with $\beta N \to 2c$,
	\[
		\frac{1}{N} \sum_{i = 1}^N \log \left|E + \frac{x}{N}- \lambda_i \right| \to \int \log|E - y| \rho_c(y) dy \quad\text{in probability.}
	\]
\end{lemma}
\begin{proof}
	With Lemma~\ref{lem:log} in mind, it suffices  to show that for $1/2 < \delta < 1$, 
	\[
		S_N : = \Ex\bigg[ \bigg| \frac{1}{N} \sum_{i = 1}^N \log|E - \lambda_i | - \log \Big|E + \frac{x}{N} - \lambda_i \Big| \bigg|^\delta \bigg] \to 0 \quad \text{as}\quad N \to \infty.
	\]
It follows from the following inequality 
\[
	|\log|u| - \log|v|| \le |u  -  v|\left( \frac{1}{|u|} + \frac{1}{|v|}\right),
\]
that 
\begin{align*}
	&\bigg| \frac{1}{N} \sum_{i = 1}^N \log|E - \lambda_i | - \log \Big|E + \frac{x}{N} - \lambda_i \Big| \bigg|^\delta \\
	&\le
	\frac{1}{N^\delta} \sum_{i = 1}^N \bigg| \log|E - \lambda_i | - \log \Big|E + \frac{x}{N} - \lambda_i \Big| \bigg|^\delta \\
	&\le \frac{x^\delta}{N^{2\delta}} \sum_{i = 1}^N \left( \frac{1}{|E - \lambda_i|^\delta} + \frac{1}{|E + \frac xN - \lambda_i|^\delta}\right).
\end{align*}
Thus, 
\[
	S_N \le \frac{x^\delta}{N^{2\delta - 1}} \left( \int \frac{1}{|E - y|^\delta} \rho_N(y) dy + \int \frac{1}{|E + \frac xN - y|^\delta} \rho_N(y) dy\right).
\]
We can bound the first integral as follows
\begin{align*}
	\int \frac{1}{|E - y|^\delta} \rho_N(y) dy &= \int_{|E-y| \le 1} \frac{1}{|E - y|^\delta} \rho_N(y) dy +  \int_{|E-y| > 1} \frac{1}{|E - y|^\delta} \rho_N(y) dy \\
	&\le D\int_{|E-y| \le 1} \frac{1}{|E - y|^\delta}  dy +  \int_{|E-y| > 1} \rho_N(y) dy \\
	&\le \frac{2D}{1-\delta} + 1,
\end{align*}
which is bounded as $N \to \infty$. Here recall that $D$ is an upper bound for the density $\rho_N(x)$ in Proposition~\ref{prop:Wegner}. The same estimate holds for the second integral. Therefore $S_N \to 0$ as $N \to \infty$. The proof is complete.
\end{proof}

\begin{proof}[Proof of Theorem~{\rm\ref{thm:local}}]
As explained before, it remains to show the condition (A).
Analogous to Lemma~\ref{lem:log-E}, it holds that under $\Prob_{\beta, N -k}$, as $N \to \infty$ with $\beta N \to 2c$,
\[
	\frac{1}{N} \sum_{i = 1}^{N - k} \log \left |E + \frac{x}{N} - \lambda_i \right| \to \int \log|E - y| \rho_c(y) dy \quad \text{in probability,}
\]
and hence, for fixed $E$, and fixed $x_1, \dots, x_k$, 
\[
	\frac{1}{N} \sum_{j = 1}^k \sum_{i = 1}^{N - k}  \log \left |E + \frac{x_j}{N} - \lambda_i \right| \to  k\int \log|E - y| \rho_c(y) dy \quad \text{in probability.}
\]
Then by the continuous mapping theorem, and the uniform integrability following from Lemma~\ref{lem:k-bound}, it follows that 
\begin{align*}
	\Ex_{\beta, N-k} \bigg[ \prod_{j = 1}^k \prod_{i = 1}^{N - k}  \left |E + \frac{x_j}{N} - \lambda_i  \right|^\beta \bigg] &= \Ex_{\beta, N-k} \left[e^{ \beta  \sum_{j = 1}^k \sum_{i = 1}^{N - k}  \log  |E + \frac{x_j}{N} - \lambda_i |} \right] \\
	&\to e^{2ck\int \log|E - y| \rho_c(y) dy}.
\end{align*}
In the formula~\eqref{k-marginal}, with $\{x_i\}_{i=1}^k$ replaced by $\{E + \frac{x_i}N\}_{i = 1}^k$, we see that 
\begin{itemize}
	\item the ratio $Z_{\beta, N - k}/Z_{\beta, N}$ converges to $(Z_c)^{-k}$, because  
		\[
			\frac{Z_{\beta, N - k}}{Z_{\beta, N}} = \prod_{i = 1}^k \frac{Z_{\beta, N - i }}{Z_{\beta, N - i + 1}} \to \frac{1}{(Z_c)^k};
		\]
	\item since all $\{x_i\}$ are distinct, the Vandermonde determinant factor converges to $1$,  
	\[
		|\Delta(\{E + \frac{x_i}N\})|^\beta = \frac{|\Delta(x)|^\beta}{N^{\beta k(k-1)/2}} \to 1;
	\]
	\item and by the continuity of $V$,
	\[
		e^{-\sum_{i = 1}^k {V(E + \frac{x_i}N)}} \to e^{-k V(E)}.
	\]
		
\end{itemize}
Therefore, as $\beta N \to 2c$,
\[
	R_N^{(k)} (x_1, x_2, \dots, x_k) \to  \frac{e^{-kV(E) + 2ck\int \log|E - y| \rho_c(y) dy}}{(Z_c)^k}  = \rho_c(E)^k,
\]
which is nothing but the condition (A). Here we have used the equation~\eqref{rhoc} charactering $\rho_c$.
The proof is complete.
\end{proof}

We conclude this paper by the following remark.
\begin{remark}
\begin{itemize}
	\item[(i)] The arguments in this paper can be generalized to show the joint convergence of $(\xi_N(E), \xi_N(E'))$ to independent homogeneous Poisson point processes with densities $\rho_c(E)$ and $\rho_c(E')$, respectively. Here $E \neq E'$ are fixed reference energies.
	
	\item[(ii)] We can also prove the following result: for fixed $E_0 \in \R$, and fixed $0< \gamma < 1$, the point processes $(\xi_N(E_0 - N^{-\gamma}), \xi_N(E_0 + N^{-\gamma} ))$ converge weakly to independent Poisson point processes with the same density $\rho_c(E_0)$.
	
\end{itemize}
\end{remark}

\bigskip
\noindent{\bf Acknowledgment.}
This work is supported by JSPS KAKENHI Grant Number JP19K14547 (K.D.T.). The authors would like to thank the referees for many helpful suggestions.

\begin{footnotesize}

\end{footnotesize}
%
%\bibliographystyle{spmpsci}
%\bibliography{bib}

\end{document}